\documentclass[11pt,reqno,tbtags]{article} 
\usepackage[utf8]{inputenc}
\usepackage[english]{babel}

\usepackage{mathrsfs}
\usepackage{amsmath}
\usepackage{amsfonts}
\usepackage{amsthm}

\usepackage{latexsym}
\usepackage[dvips]{graphics,color}
\usepackage{graphicx} 

\usepackage{amssymb}
\usepackage{units}
\usepackage{mathrsfs}
\numberwithin{equation}{section}
\newtheorem{theorem}{Theorem}
\newtheorem{lemma}{Lemma}

\newtheorem{corr}{Corollary}
\newtheorem{rem}{Remark}

\newtheorem{deff}{Definition}

\newtheorem{exempel}{Example}

\newcommand{\fa}{\mathscr F}
\newcommand{\ex}{\mathbb E}
\newcommand{\pr}{\mathbb P}
\newcommand{\real}{\mathbb R}
\newcommand{\bigo}{\mathcal O}
\newcommand{\lilo}{ \mbox{{\tiny $\mathcal O$}}  } 
\newcommand{\ind}[1]{\textbf I_{#1}}
\newcommand{\limd}{\stackrel{\text d}{\to}}
\newcommand{\rymd}{\vspace{0.2cm}}
\newcommand{\e}{\mathcal E}
\newcommand{\mand}{\quad\mbox{and}\quad}
\newcommand{\sgn}{\mathrm{sgn}}

\newenvironment{romenumerate}{\begin{enumerate}
 }{\end{enumerate}}

\title{\textsc{Limit theorems for stochastic approximation algorithms}}
\author{Henrik Renlund}

\begin{document}

\pagenumbering{roman}

\maketitle

\begin{abstract}
 We prove a central limit theorem applicable to one dimensional stochastic approximation algorithms
that converge to a point where the error terms of the algorithm do not vanish. We show 
how this applies to a certain class of these algorithms that in particular covers a generalized Pólya urn model,
which is also discussed.  In addition, we  show how to scale these algorithms in 
some cases where we cannot determine the limiting distribution but expect it to be non-normal.
\end{abstract}

\vspace{2cm}

\tableofcontents

\newpage

$\phantom.$

\thispagestyle{empty}

\newpage

\pagenumbering{arabic}
\setcounter{page}{1}

\section{Introduction and preliminaries} \label{prel}
The following paper is a continuation of the work \cite{lic}, which deals with convergence of
 stochastic approximation algorithms, as defined in Definition \ref{def} below. A stochastic approximation
algorithm may be said to be a stochastic process that on average follows a solution curve to an ordinary differential
equation. One may consult e.g.\ \cite{Ben99} for a concise treatment along this line of thought.  

\begin{deff}\label{def} $\phantom{1}$\\
A stochastic approximation algorithm $\{X_n\}$ is a 
stochastic process taking values in $[0,1]$ and adapted
to the filtration $\{\fa_n\}$ that satisfies
\begin{equation}\label{defeq}
  X_{n+1}-X_n=\gamma_{n+1}[f(X_n)+U_{n+1}],
\end{equation}
where $\gamma_n,U_n\in\fa_n, f:[0,1]\to\real$ and 
the following conditions hold a.s.
\begin{itemize}
 \item[(i)] $c_l/n\leq \gamma_n\leq c_u/n$,
 \item[(ii)] $|U_n|\leq K_u$,
 \item[(iii)] $|f(X_n)|\leq K_f$, and
 \item[(iv)] $|\ex_n(\gamma_{n+1}U_{n+1})|\leq K_e/n^2$,
\end{itemize}
where the constants $c_l,c_u,K_u,K_f,K_e$ are positive real numbers and 
$\ex_n(\cdot)$ denotes the conditional expectation $\ex(\cdot|\fa_n)$.
\end{deff}
As is shown in \cite{lic}, if the drift function  $f$ is continuous,
the limit of such a process always exists and is contained
in the zero set of $f$, i.e.\ the set $\{x: f(x)=0\}$. Certain zeros can be excluded from
the set of possible limit points, in particular the unstable ones (under additional 
assumptions\footnote{It is required that the variance of the error terms $U_n$ does not
vanish at this point. Note that the unstable zeros are only excluded in the sense that
the probability of convergence to any specific point is zero. Hence, if there are uncountably
many unstable points - if e.g.\ $f\equiv 0$ on an interval - then all we can deduce is that
there are no point masses at these points.}). 
A zero $x_u$ is said to be \emph{unstable} if the drift 
locally points away from, or is zero at,  this point, i.e.\ that $f(x)(x-x_u)\geq 0$ when $x$ is near $x_u$. 
On the contrary there is a positive probability (under additional assumptions) that the process ends up at 
a \emph{stable} zero $x_s$, i.e.\ a point where $f(x)(x-x_s)< 0$, when $x\neq x_s$ is near $x_s$, 
so that the drift locally is  pointing towards it. 

We will throughout think of this process as having a stable zero at $p$, and typically that 
$f$ is differentiable at this point. Then  \[ f(x)=-h(x)(x-p) \] where $h$ is
continuous at $p$ and $h(x)> 0$, when $x\neq p$ is close to $p$.

This paper investigates how to scale $X_n-p$ to get convergence to some non-trivial 
distribution.  Section \ref{l} contains some necessary tools. Theorem \ref{st} in Section \ref{aclt} is a  
a central limit theorem for class of processes that 
covers stochastic approximation algorithms (as defined above). Section \ref{atsaa} and \ref{genpol} show 
how this applies to stochastic approximation algorithms and an urn model, respectively. The urn model
will be specified at the end of this section. Theorem \ref{genlim} in Section \ref{alt} provides a limit theorem for stochastic
approximation algorithms, although we can not identify the limiting distribution -- there is a brief
discussion of this problem in Section \ref{anotponld}. Section \ref{limappgenpol} is concerned with 
the application of Theorem \ref{genlim} to the urn model described below. 

The proper scaling and limit of $X_n-p$ turns out to depend largely on the limit 
$\hat\gamma = \lim_n n\gamma_nh(X_{n-1})$. When it exists, $\hat\gamma>1/2$ and $\hat\gamma=1/2$ 
are associated to a central limit theorem (Theorem \ref{st}), although with different scaling in the respective
cases, whereas
$\hat\gamma\in (0,1/2)$ to is associated with convergence to some unidentified distribution (Theorem \ref{genlim}).
We have not studied what happens when $\hat\gamma=0$ or when this limit does not exist, see Remarks 
\ref{limnoll} and \ref{limnej} in Section \ref{anotponld} for comments on these respective cases. 

The application to be discussed is the following generalized P\'olya urn model.
Consider an urn with balls of two colors, white and black say. Let $W_n$ and $B_n$ denote
the number of balls of each color, white and black respectively, 
 after the $n$'th draw and consider the initial values $W_0=w_0>0$ and
$B_0=b_0>0$ to be fixed. After each draw we notice the color and replace it along with
additional balls according to the \emph{replacement matrix}
\begin{align}  & \quad\; \text{W}  \,\;\; \text{B} \nonumber \\
\begin{array}{c}
\text{W}\\ \text{B}
\end{array} &
\left( \begin{array}{c c}
a & b \\ c & d 
       \end{array} \right), \label{repmat}
 \end{align}
where $\min\{a,b,c,d\}\geq 0$. The replacement matrix (\ref{repmat}) should be interpreted as; if a
white ball is drawn it is replaced along with an additional $a$ white and $b$ black balls. If a black
ball is drawn it is replaced along with an additional $c$ white and $d$ black balls. 
Let $T_n=W_n+B_n$ denote the total number of balls at time $n$. As shown in Section 3.1 of \cite{lic},
 if $\min\{a+b,c+d\}>0$ then the fraction of white balls $X_n=W_n/T_n$ is a stochastic approximation
algorithm with
\begin{align*} 
  f(X_n) &= \ex_n[T_{n+1}(X_{n+1}-X_n)]\quad \mbox{given by} \\
  f(x) &=\alpha x^2+\beta x+c,\quad\mbox{where}\quad\alpha=c+d-a-b\quad\mbox{and}\quad\beta=a-2c-d, \\
 \gamma_{n+1} &=1/T_{n+1} \quad\mbox{and}\\
 U_{n+1}&=T_{n+1}(X_{n+1}-X_n)-f(X_n).
\end{align*}
Another quantity of interest is the second moment of the error terms $U_n$. 
This turns out to be a polynomial $\e$ in $X_n$, $\e(Z_n)=\ex_n U_{n+1}^2$, that is given by 
\[ \e(x)=x(1-x)[a-c+\alpha x]^2. \] 
It is know, e.g.\ from Theorem 6 of \cite{lic}, that - apart from the case when the replacement matrix
is a multiple of the identity matrix (in which case $X_n$ converges to a beta distribution) - the
limit $\lim X_n$ has a one point distribution at the unique stable zero of $f$, which we 
will denote by $p$. Recall that $h$ is defined via the relation $f(x)=-h(x)(x-p)$. As $f$ is a polynomial 
and thus infinitely differentiable at $p$ with $f'(p)<0$, $h$ will be differentiable 
at $p$ as well as positive near $p$.  

Notice that as $X_n\to p$ we have $T_n/n\to (a+b)p+(c+d)(1-p)=:\gamma^{-1}$ so that in particular
$n\gamma_n\to \gamma$. Since $X_n\to p$, continuity of $h$ implies that $h(X_n)\to h(p)$ and
thus $\hat\gamma_n = n\gamma_nh(X_{n-1})\to \gamma h(p)=-\gamma f'(p)$.

\subsection{Lemmas}\label{l}

Let $Q_n=X_n-p$ and rewrite (\ref{defeq}) to  
\begin{align}
 X_{n+1}-p&=X_n-p+\gamma_{n+1}[f(X_n)+U_{n+1}] 
=[1-\gamma_{n+1}h(X_n)]Q_n+\gamma_{n+1}U_{n+1} \nonumber \\
\Rightarrow \quad Q_{n+1}&=\left[1-\frac{\hat\gamma_{n+1}}{n+1}\right]Q_n+\frac{\hat U_{n+1}}{n+1}, \label{q:et}
\end{align}
where 
\[ \hat\gamma_{n}=n\gamma_{n}h(X_{n-1})\quad\mbox{and}\quad
\hat U_{n}=n\gamma_{n}U_{n}.\]

Equation (\ref{q:et}) explains  our interest in recursive sequences of the following form:
 \[ b_{n+1}=(1-A/n)b_n+B/n. \]
It is easy to see that 
$b_{n+1}=b_n$ if and only if $b_n=B/A$.  

The following lemma deals with slightly more
general recursions and is a  modification of Lemma 4.2 of \cite{Fab67}, which in turn
is a summary of Lemmas 1-4 of \cite{Chung}. Like \cite{Fab67} we will refer to it as Chung's lemma.
\begin{lemma}[Chung]
Let $b_n, A_n, B_n, D_n, g_n$ be real numbers such that 
the following holds
\[ b_{n+1}=(1-A_n/g_n)b_n+B_n/g_n+D_n, \]
and with the following properties
\begin{align*}
 0 <a_0 &=\liminf_{n\to\infty} A_n\leq a_1=\limsup_{n\to\infty}A_n<\infty, \\
 B_n\to B&\geq 0, \quad\quad 0<g_n\to \infty,\quad \sum_{n} 1/g_n=\infty.
\end{align*}
\begin{romenumerate}
 \item If $D_n\leq 0$ then $\displaystyle\limsup_{n\to\infty} b_n\leq B/a_0$.
 \item If $D_n\geq 0$ then $\displaystyle\liminf_{n\to\infty} b_n\geq B/a_1$.
\end{romenumerate}
As a consequence, if $D_n\equiv 0$ and $\lim_n A_n$ exists and equals $A>0$, then $\lim_n b_n$
exists and equals $B/A$.
\end{lemma}

\begin{proof}
To prove the first part, assume $D_n\leq 0$ and let $\epsilon\in(0,a_0)$ be an arbitrarily
small number. Let
$N_0$ be large enough to ensure that $n\geq N_0$ implies 
\[ \quad a_0-\epsilon\leq A_n\leq a_1+\epsilon, \quad  B_n\leq B+\epsilon \quad\mbox{and}\quad
g_n\geq\max\left\{\frac{B+\epsilon}{\epsilon},a_1+\epsilon\right\}  \] 
Now, suppose that for some $m\geq N_0$ we have $b_m\geq U=(B+2\epsilon)/(a_0-\epsilon)$.
If no such $m$ exists for any $\epsilon$ there is nothing to prove. Then $b_m$ is positive and
\begin{align*}
 b_{m+1}&=b_m-\frac{A_m}{g_m}\cdot b_m+\frac{B_m}{g_m}+D_m \\
&\leq  b_m-\frac{a_0-\epsilon}{g_m}\cdot\frac{B+2\epsilon}{a_0-\epsilon}+\frac{B+\epsilon}{g_m}+0
=b_m-\frac{\epsilon}{g_m}
\end{align*}
and hence $b_{m+1}\leq b_m-\epsilon/g_m$. If $b_{m+1},\ldots,b_{m+k-1}\geq U$, then 
\[ b_{m+k}\leq b_m-\epsilon\sum_{n=m}^{m+k-1}\frac 1{g_n}. \] Since $\sum_{n\geq m} 1/g_n$ diverges
to infinity there must be a $k$ such that $b_{m+k}< U$. 
Next, notice that if $b_n<U$, for $n\geq N_0$, we have 
$b_{n+1}\leq (1-A_n/g_n)b_n+(B+\epsilon)/g_n<U+\epsilon$.
In conclusion; if $b_n$, after $N_0$,  is above
$U$ it will decrease to a value below $U$ and then it may never again, in a single step, exceed $U$ by more
than $\epsilon$. Since $\epsilon$ is arbitrary we conclude that $\limsup_n b_n\leq B/a_0$.

To prove (ii), assume $D_n\geq 0$ and suppose first that $B>0$ and let 
$\epsilon>0$ be arbitrarily small and $\epsilon<\min\{B/2,a_0\}$.
Similar to the proof of (i), if we suppose that $b_m\leq L=(B-2\epsilon)/(a_1+\epsilon)$,
for some $m\geq N_1$, where $N_1$ is large enough to ensure that 
$n\geq N_1$ implies 
\[ a_0-\epsilon\leq A_n\leq a_1+\epsilon,\quad B_n\geq B-\epsilon, \quad\mbox{and}\quad
g_n\geq a_1+\epsilon, \]
then $b_{m+1}\geq b_m+\epsilon/g_m$ so that $b_{m+k}>L$ for some $k\geq 1$. For $n\geq N_1$,
if $b_n>L$ then 
\begin{align*}
 b_{n+1}>\left(1-\frac{a_1+\epsilon}{g_n}\right)L+\frac{B-\epsilon}{g_n}=L+\frac{\epsilon}{g_n},
\end{align*}
so that in fact all values $b_{n+k}$ stays above $L$. We conclude that
$\liminf_n b_n\geq B/a_1$.

If $B=0$ pick $\epsilon<a_0$. If $b_m<-P=-2\epsilon/(a_0-\epsilon)$, for some $m\geq n$, where $n\geq N_1$ 
($N_1$ as above), then $b_{m+1}\geq b_m+\epsilon/g_m$.
 Hence there is $k\geq1$ such that 
$b_{m+k}>-P$. Now, if $b_n>-P$ then $b_{n+1}>-P+\epsilon/g_n$.
Hence $\liminf_n b_n\geq 0$.
\end{proof}

The following corollary, specific for $g_n=n$, can be found in \cite{Fabian}. 
\begin{corr}\label{Vaclav}
 Suppose $A>0$, $h_n, b_n$ are real numbers,
\begin{equation}\label{rec2}
  b_{n+1} =(1-A/n)b_n+h_n/n. 
\end{equation}
Then 
\[ b_n\to 0\quad\mbox{if and only if}\quad \bar h_n=\frac 1n\sum_1^n h_k\to 0. \]
\end{corr}

\begin{proof}$\phantom.$\\
''$\Longrightarrow$'' Assume $b_n\to 0$. Rewrite (\ref{rec2}) to
$n(b_{n+1}-b_n)+Ab_n=h_n$ so
\begin{equation}\label{rel}
\sum_{j=1}^n h_j=\sum_{j=1}^n j(b_{j+1}-b_j)+A\sum_{j=1}^nb_j=(A-1)\sum_{j=1}^nb_j+nb_{n+1}. 
\end{equation}
Since $b_n\to 0$ implies $\bar b_n=\frac1n\sum_1^n b_j\to 0$, this shows the necessity of
$ \bar h_n\to 0$.

\vspace{0.2cm}
\noindent ''$\Longleftarrow$'' Assume $\bar h_n\to0$.  Notice that from (\ref{rel})  we have that
$b_{n+1}=(1-A)\bar b_n+\bar h_n$ and so it suffices to show $\bar b_n\to 0$. To that end, we calculate
\begin{align*}
 \bar b_{n+1} &= 
\frac n{n+1}\bar b_n+\frac 1{n+1}b_{n+1}
=\left(1-\frac1{n+1}\right)\bar b_n+\frac{(1-A)\bar b_n+\bar h_n}{n+1}\\
&=\left(1-\frac A{n+1}\right)\bar b_n+\frac 1{n+1}\bar h_n.
\end{align*}
Chung's lemma yields $\lim_n \bar b_n=0$.
\end{proof}

\begin{lemma}\label{prodlemma}
 Suppose $n$ and $m$ are integers, $n>m\geq 1$, and that $\alpha\in(0,1)$ is fixed. Then 
\[ P_\alpha(m,n):=\prod_{k=m}^n\left(1-\frac \alpha k\right)=\left(\frac mn\right)^\alpha(1+\bigo(m^{-1})). \]
Also, if  we consider $m$ fix, then $n^\alpha P_\alpha(m,n)$ converges as
$n$ tends to infinity.
\end{lemma}
\begin{proof}
The first part follows from $1-\frac \alpha k =\exp\left\{-\frac \alpha k+\bigo(k^{-2})\right\}$, 
$\sum_{k\geq n}\frac 1{k^2}$ $=\bigo(n^{-1})$ and $\sum_{k=m}^n \frac 1k$ $=\ln\left(\frac nm\right)+\bigo(m^{-1})$.
As $n^\alpha P_\alpha(m,n)$ is increasing (in $n$),
 the second fact follows from the boundedness provided by the first part and monotonicity. 
\end{proof}

\begin{lemma}\label{Venter}
 Suppose $b_n$ is a sequence of non-negative  numbers such that
\begin{equation}\label{theB} b_{n}\leq \left(1-\frac{A}{n}\right)b_{n-1}+\frac{B}{n^{1+p}}, \end{equation}
where $B>0$ and $p>A$. Then $ b_n=\bigo(n^{-A})$. 
 \end{lemma}

\begin{rem}
  Lemma \ref{Venter} exists in a stronger form, without proof, in \cite{Ven66}.
\end{rem}

\begin{proof}
Suppose that $m>A$. Then we get, from first iterating the inequality (\ref{theB}) and then applying Lemma \ref{prodlemma}, 
\begin{align*}
  b_{n} &\leq b_m\prod_{k=m+1}^{n}(1-Ak^{-1})+B\sum_{k=m}^{n}k^{-(1+p)}\prod_{j=k+1}^{n}(1-Aj^{-1}) \\
&\leq n^{-A}\bigg[b_mm^{A}(1+\bigo(m^{-1})+B\sum_{k=m}^nk^{-(1+p-A)}(1+\bigo(k^{-1}))\bigg],
 \end{align*}
where the last sum, being convergent, has an upper bound independent of $n$. 
\end{proof}

\begin{deff}[Definition 6.1.2 of \cite{Gut}]
$\{X_n\}$ and $\{Y_n\}$ are said to  be \emph{distributionally equivalent} if
\[ \pr(X_n\neq Y_n)\to 0,\quad n\to \infty. \] 
\end{deff}

\begin{lemma}\label{distcon}
 If $\{X_n\}$ and $\{Y_n\}$ are distributionally equivalent and if 
$X_n\limd X$, then $Y_n\limd X$.
\end{lemma}

\begin{proof}
 Appears e.g.\ as Theorem 6.1.2 (ii) of \cite{Gut}.
\end{proof}

\begin{lemma}[Egoroff's theorem]
Let $(\Omega,\fa,\pr)$ be a probability space.
Let $X$ and $X_1,X_2,\ldots$ be random variables $\Omega\to \real$ and suppose that 
$X_n\to X$ a.s. Then for every $\epsilon>0$ there exists 
$B_\epsilon\in\fa$ such that $\pr(B_\epsilon^c)<\epsilon$ and
$X_n\to X$ uniformly on $B_\epsilon$.
\end{lemma}

\begin{proof}
 This is a special case of Proposition 3.1.3 of \cite{Cohn}. 
\end{proof}

\begin{corr}\label{egocorr}
 Let $(\Omega,\fa,\pr)$ be a probability space and let 
$\fa_n$ be a filtration. Suppose that $X_n$ is an adapted sequence such 
that $X_n\to x\in\mathbb R$ a.s. Then for every $\epsilon>0$ there exists
a set $B_\epsilon$ and an adapted sequence $\hat X_n$ such that $X_n$ equals
$\hat X_n$ on $B_\epsilon$, $\pr(B_\epsilon^c)<\epsilon$, and
 $\hat X_n$ converges uniformly to $x$.
\end{corr}

\begin{proof}
Given $\epsilon>0$, Egoroff's theorem gives us a set $B_\epsilon$ on
which $X_n$ converges uniformly to $x$, i.e.\ for every $\delta>0$ there is
an $N$ such that 
\[ \sup_{n\geq N}|X_n\ind{B_\epsilon}-x|<\delta. \]
Let $N_0=0$. 
For $n\geq 1$, define
 $N_{n}>N_{n-1}$ to be such that 
$\sup_{k\geq N_{n}}|X_n\ind{B_\epsilon}-x|<1/n$. Note that $N_n$
does not depend on $\omega$, and hence we can define the adapted
$\hat X_n$ via 
\[ \hat X_n(\omega)=\begin{cases}
             X_n(\omega), & \mbox{if}\; n\leq N_1 \\
             X_n(\omega), & \mbox{if}\;n\in\{N_k+1,\ldots,N_{k+1}\} \;\mbox{and}\; |X_n(\omega)-x|< 1/k, \\
             x, & \mbox{if}\;n\in\{N_k+1,\ldots,N_{k+1}\} \;\mbox{and}\; |X_n(\omega)-x|\geq 1/k 
            \end{cases} \]
Then, for every $\omega$,
$\hat X_n(\omega)$ converges uniformly to $x$, since given any $\delta>0$ we have
$\sup_{n>N_m}|\hat X_n(\omega)-x|<\delta$, if $m=\lceil 1/\delta\rceil$.
Moreover, for every $\omega\in B_\epsilon$, $X_n$ and $\hat X_n$ agree.
\end{proof}

\begin{lemma}\label{remma1}
Given a stochastic process $\{X_n\}$, suppose that for every $\epsilon>0$
we can find  a process $\{Y_{n,\epsilon}\}$ such that
\begin{romenumerate}
 \item $\pr\{\omega: X_n(\omega)\neq Y_{n,\epsilon}(\omega)$ for some $n \}<\epsilon$  and
 \item $Y_{n,\epsilon}\limd Y$, as $n\to\infty$, for all $\epsilon$.
\end{romenumerate}
Then $X_n\limd Y$. 
\end{lemma}

\begin{proof}
Choose a sequence $\epsilon_n$, tending to $0$ as $n$ tends to infinity, in such a way 
that the distribution function of $Y_{n,\epsilon_n}$ tends to that of $Y$. Then 
$\pr(X_n\neq Y_{n,\epsilon_n})<\epsilon_n\to 0$ so that
$X_n$ and $Y_{n,\epsilon_n}$ are distributionally equivalent. Since $Y_{n,\epsilon_n}\limd Y$, we
also have $X_n\limd Y$ by Lemma \ref{distcon}.
\end{proof}

We will also need the following consequence of the martingale convergence theorem.
\begin{lemma} \label{improve}
Let $(\Omega,\fa,\pr)$ be a probability space and let 
$\fa_n$ be a filtration. Suppose that $Y_n$ is an adapted sequence such 
that a.s.\ 
\begin{equation*}
 \sum_{k=1}^\infty \ex Y_k^2\leq C_1<\infty
\quad\mbox{and}\quad
 \sum_{k=1}^\infty|\ex_{k-1} Y_k|\leq C_2<\infty.
\end{equation*}
Then $\sum_{1}^n Y_k$ converges a.s.
\end{lemma}

\begin{proof}
 Define the martingale
\[ S_n=\sum_{k=1}^n(Y_k-\ex_{k-1}Y_k). \]
$S_n$ is in $L_2$ since $\ex S_n^2\leq\sum_1^n\ex Y_k^2\leq C_1$ and hence a.s.\ convergent. 

The sum $T_n=\sum_1^n\ex_{k-1}Y_k$ is a.s.\ convergent, since 
it is absolutely convergent by assumption. 
Since $\sum_{1}^n Y_k$ is the sum of $S_n$ and $T_n$, it must also be a.s.\ convergent.
\end{proof}

\begin{lemma} (A version of Kronecker's lemma) \label{Kron}
 Let $a_k$ be a sequence of reals. Let $0<b_1\leq b_2\leq \ldots \leq b_n\to \infty$. 
Set 
\[ S_n=\sum_{k=1}^n a_k \quad\mbox{and}\quad T_n=\sum_{k=1}^n b_ka_k. \]
Assume $S_n\to s\in\mathbb R$. Then $T_n/b_{n+1}\to 0$.
\end{lemma}

\begin{proof}
We may rewrite $T_n/b_{n+1}$ as
\begin{align}
 \frac {T_n}{b_{n+1}} &=S_n-\frac1{b_{n+1}}\sum_{k=1}^n(b_{k+1}-b_k)S_k. \label{a1} 
\end{align}
Fix an $\epsilon>0$. By convergence of $S_n$ we know that there is an $N$ such that
$k\geq N$ implies $|S_k-s|<\epsilon$. Assume $n\geq N$ and continue 
\begin{align*}
 (\ref{a1}) &= S_n
-\underbrace{\frac{ \sum_{N}^n(b_{k+1}-b_k)s}{b_{n+1}}}_{A_n}
-\underbrace{\frac{ \sum_{N}^n(b_{k+1}-b_k)(S_k-s)}{b_{n+1}}}_{B_n}
-\underbrace{\frac{  \sum_{1}^{N-1}(b_{k+1}-b_k)S_k}{b_{n+1}}}_{C} \\
\end{align*}
$S_n$ and $A_n$ will cancel out as $n\to \infty$, since $(b_{n+1}-b_N)/b_{n+1}\to 1$.   
$B_n$ is bounded by $\epsilon(b_{n+1}-b_N)/b_{n+1}$. $C$ is a something finite
divided by $b_{n+1}$, so it tends to 0. Hence, $T_{n}/b_{n+1}\to 0$. 
\end{proof}

\newpage

\section{A central limit theorem} \label{aclt}
The first results on asymptotic normality in stochastic approximation was for the Robbins-Monro
procedure (see \cite{RM51}) in \cite{Chung}. The methods of that paper was extended in 
\cite{Der56} and \cite{Bur56} to the Kiefer-Wolfowitz procedure (see \cite{KW52}). See also
\cite{Sac58} for a different approach.  

The following (one dimensional) theorem, and its proof, is an adaptation of
(the multidimensional) Theorem 2.2 of \cite{Fabian}.
The main adaptation is to allow for general step length sequences $1/g_n$ instead
of $1/n^\alpha$. This allows us in applications to establish asymptotic normality 
for cases where the normalizing sequence is $\sqrt n$
as well as cases where it is $\sqrt{n/\ln n}$. 
\begin{theorem} \label{st}
Suppose $\{Z_n,n\geq 1\}$ is a stochastic process
adapted to a filtration $\{\fa_n,n\geq1\}$,
such that 
\begin{equation}\label{simple}
 Z_{n+1}=(1-\Gamma_{n+1}/g_n)Z_n+V_{n+1}/\sqrt {g_n},
\end{equation}
where $0<g_n\to \infty$, $\sum_n 1/g_n=\infty$,
and the $\Gamma_n,V_n\in\fa_n$ are such that a.s.\
\begin{equation}
 \ex_n V_{n+1}=\lilo(1/\sqrt{g_n}),\quad \ex_n V_{n+1}^2\to \sigma^2,\quad
 \ex_n V_{n+1}^2\leq C_V\quad{and}\quad \Gamma_n\to \Gamma \label{krav}
\end{equation}
for some (strictly) positive and deterministic $\sigma^2, C_V$ and $\Gamma$. 
If
\begin{equation}\label{kraf}
 \lim_{n\to \infty}\ex\left[V_{n+1}^2\ind{\{V_{n+1}^2\geq \epsilon g_n\}}\right]=0,\quad
\mbox{for all $\epsilon>0$,}
\end{equation}
then \[ Z_n\limd \mbox{\emph{N}} \left(0,\;\frac{\sigma^2}{2\Gamma}\right). \]
In the particular case  $g_n=n$  (\ref{kraf}) can be relaxed to
\begin{equation} \label{weakkraf}
 \lim_{n\to\infty}\frac 1n\sum_{k=1}^n\ex\left[V_{k+1}^2\ind{\{V_{k+1}^2\geq \epsilon k\}}\right]=0,\quad
\mbox{for all $\epsilon>0$,}
\end{equation}
with the same conclusion.
\end{theorem}

\begin{proof} 
First of all, let us show why N$(0,\sigma^2/2\Gamma)$ is a good candidate for the limiting
distribution, if such exists. To do so, let us assume that  $Z_1=0$, $\Gamma_n\equiv \Gamma$
and that $V_1,V_2,\ldots$ 
are i.i.d.\ N$(0,\sigma^2)$, i.e.\ normally
distributed with mean $0$ and variance $\sigma^2$.
 Then 
\begin{align*}
 Z_1 &=0,\\
 Z_2 &=V_2/\sqrt{g_1}, \\
 Z_3 &=(1-\Gamma/g_2)V_2/\sqrt{g_1}+V_3/\sqrt{g_2}, \\
 Z_4 &=(1-\Gamma/g_3)(1-\Gamma/g_2)V_2/\sqrt{g_1}+(1-\Gamma/g_2)V_3/\sqrt{g_2}+V_4/\sqrt{g_3},\;\; \mbox{etc.}
\end{align*}
Hence,  $Z_n$ is a linear combination of independent normally distributed random variables and hence it is also
also normally distributed. Let $b_n=\ex Z_n^2$. Squaring (\ref{simple}) 
and taking expectation  yields, since  $Z_n$ is independent of $V_{n+1}$ and
$\ex V_n\equiv 0$,
\[ b_{n+1}=(1-\Gamma/g_n)^2 b_n+\sigma^2/g_n=(1-A_n/g_n)b_n+\sigma^2/g_n, \]
where $A_n=2\Gamma-\Gamma^2/g_n\to 2\Gamma$. An application of Chung's lemma yields
$\lim_n b_n=b=\sigma^2/2\Gamma$. If $Z\in$ N$(0,b)$, then
$F_{Z_n}(x)\to F_Z(x)$, for every $x$, so that $Z_n\limd Z$.

The remainder of the proof is organized in two parts as follows; in the first part we impose stronger assumptions 
than those of the theorem
and show that the desired result is true. Then in the second part, we justify why  these stronger
assumptions make no difference to the result.

\rymd
\noindent\textbf{Part 1:} Here we assume that $Z_1=0$, $\Gamma_n\equiv \Gamma$ and $\ex_n V_{n+1}\equiv 0$.

Let
\begin{align}
\varphi_n(t) &=\ex e^{itZ_n},  \quad\mbox{i.e.\ the characteristic function of $Z_n$}, \nonumber \\
B_n &=1-\Gamma/g_n, \nonumber \\
\psi_1(t)&=1\quad\mbox{and}\quad \psi_{n+1}(t)=\psi_n(B_nt)(1-t^2\sigma^2/2g_n),\quad n\geq 1. \label{psi}
\end{align}
Now, consider the following.
\[ \mbox{\textbf{Claim:}}\quad\varphi_n(t)-\psi_n(t)\to 0,\quad\mbox{as $n\to\infty$, for all $t$.} \]

Suppose that this is true. Then, as $\psi_n$ does not depend on the actual distribution
of $V_n$, we may choose any distribution on the $V_n$:s to calculate $\varphi_n$ in order to
determine the limit of $\psi_n$. If
$V_n$ are i.i.d.\ N$(0,\sigma^2)$, then we know from the discussion above that $\varphi_n(t)\to 
\psi(t)=e^{-\frac 12t^2\sigma^2/2\Gamma}$, i.e.\ the characteristic function of a
N$(0,\sigma^2/2\Gamma)$ variable. Especially this implies that
$\lim_n \varphi_n(t)=e^{-\frac 12t^2\sigma^2/2\Gamma}$, regardless of the distribution on $\{V_n\}$,
and this is equivalent to $Z_n\limd$ N$(0,\sigma^2/2\Gamma)$.

To show that the claim is true, note that, from (\ref{simple}) and (\ref{psi}),
\begin{align}\label{lang}
 |\varphi_{n+1}(t) -\psi_{n+1}(t)|&=\big|\ex e^{itB_nZ_n+itV_{n+1}/\sqrt{g_n}}-\psi_n(B_nt)(1-t^2\sigma^2/2g_n)\big| 
\nonumber \\
&=\bigg |\ex\bigg\{ \left[e^{itB_nZ_n}-\psi_n(B_nt)\right]\left(1-t^2\sigma^2/2g_n\right) 
\nonumber \\
&\quad\quad+ e^{itB_nZ_n}\left(e^{itV_{n+1}/\sqrt{g_n}}-1+t^2\sigma^2/2g_n\right)\bigg\}\bigg| 
\nonumber \\
&\leq |1-t^2\sigma^2/2g_n|\cdot|\varphi_n(B_nt)-\psi_n(B_nt)| 
\nonumber \\
&\quad\quad+ \underbrace{\ex \big|\ex_n e^{itV_{n+1}/\sqrt{g_n}}-1+t^2\sigma^2/2g_n}_{=\zeta_n}\big|,
\end{align}
where the last step comes from smoothing and the fact that $|e^{itB_nZ_n}|\leq 1$. 
Next, we examine $\zeta_n$ (as defined in (\ref{lang})),
\begin{align*}
 \zeta_n &= \ex \left| \ex_n e^{itV_{n+1}/\sqrt{g_n}}-1+\frac{t^2}{2g_n}\ex_n V_{n+1}^2 
+\frac{t^2}{2g_n}\ex\left[\sigma^2-\ex_n V_{n+1}^2\right] \right|.
\end{align*}
The following inequality, which appears e.g.\ as  Lemma A.1.2 of \cite{Gut} where a proof can be found, will prove useful. 
For any real $v$ and integer $m\geq 0$,  
\begin{equation}\label{expineq} \left|e^{iv}-\sum_{k=0}^m \frac{(iv)^k}{k!}\right|\leq \min\left\{ \frac{2|v|^m}{m!},
\frac{|v|^{m+1}}{(m+1)!}\right\}. \end{equation}
We are going to show that $|\varphi_n(t)-\psi_n(t)|$ tends to zero for any $t$. We fix an arbitrary $T>0$ and 
consider $|t|\leq T$. Choose an $\epsilon'>0$ and put 
$\epsilon=(6\Gamma\epsilon'/C_VT^2)^2$. 
 By using the triangle inequality, thus splitting $\zeta_n$ into two parts, 
and then applying inequality 
(\ref{expineq}) with $m=2$ on the first part, again splitting into two cases  depending on whether $V_{n+1}^2$ is above 
or below $\epsilon g_n$, we get
\begin{align}
 \zeta_n&\leq \ex \left|\ex_n e^{itV_{n+1}/\sqrt{g_n}}-1+\frac{t^2}{2g_n}\ex_n V_{n+1}^2 \right|
+\frac{t^2}{2g_n}\ex\big|\sigma^2-\ex_{n}V_{n+1}^2\big|  \nonumber \\
&\leq \ex\min\left\{\frac{|tV_{n+1}|^2}{g_n},\frac{|tV_{n+1}|^3}{6g_n^{3/2}} \right\}+\frac{t^2}{2g_n}\ex
\big|\sigma^2-\ex_{n}V_{n+1}^2\big|  \nonumber \\
&\leq \ex\left[\frac{|tV_{n+1}|^2}{g_n}\ind{\{V_{n+1}^2\geq \epsilon g_n\}}\right] 
+\ex\left[\frac{|tV_{n+1}|^3}{6g_n^{3/2}}\ind{\{V_{n+1}^2<\epsilon g_n\}}\right]
+t^2\frac{\ex\big|\sigma^2-\ex_{n}V_{n+1}^2\big|}{2g_n} \nonumber \\
&\leq t^2\frac{\ex\big[V_{n+1}^2\ind{\{V_{n+1}^2\geq \epsilon g_n\}}\big]}{2g_n} 
+|t|^3\frac{\sqrt\epsilon C_V}{6g_n}+t^2\frac{\ex\big|\sigma^2-\ex_{n}V_{n+1}^2\big|}{2g_n}\nonumber  \\
&=|t|h_n(t)/g_n, \label{h}
\end{align}
where we by the last equality define a function $h_n(t)$. Two things in particular are to be noted
about this function. First, as $n\to \infty$,
\begin{equation}\label{h_et}
 h_n(t)=\frac{|t|}{2} \ex\left[V_{n+1}^2\ind{\{V_{n+1}^2\geq \epsilon g_n\}}\right]
+ t^2\sqrt\epsilon C_V/6+|t|\lilo(1),
\end{equation}
so that by assumption (\ref{kraf}) we have, for any fixed $t$
\begin{equation}\label{sqrte} \lim_{n\to\infty} h_n(t)=t^2\sqrt\epsilon C_V/6 \end{equation}
Secondly,  $h_n(t)$ is increasing in $|t|$. 
so that
\begin{equation}
\label{xkcd} h_n(s)\leq h_n(t), \quad\mbox{if $|s|\leq |t|$}.
\end{equation}

Now, let $b_n(t)=|\varphi_n(t)-\psi_n(t)|$. From (\ref{lang}) and (\ref{h}) we conclude that
\begin{equation} \label{bineq}
b_{n+1}(t)\leq |1-t^2\sigma^2/2g_n| b_n(B_nt)+|t|h_n(t)/g_n. 
 \end{equation}
We want to show that $b_n(t)$ tends to zero for any $|t|\leq T$. 
We will consider  indices $n$ larger than $N$, where $N$ is such that
$n\geq N$ implies $g_n\geq \max\{T^2\sigma^2/2,\Gamma\}$ and hence that $B_n=1-\Gamma/g_n\in(0,1)$ and, 
from (\ref{xkcd}) and (\ref{bineq}),  that
\begin{equation} \label{bineq2}
b_{n+1}(t)\leq  b_n(B_nt)+|t|h_n(T)/g_n.
 \end{equation}

First, notice that
\[ b_1(t)=|\varphi_1(t)-\psi_1(t)|=|\ex e^{itZ_1}-1|\leq \ex \min\{2,|tZ_1|\}, \]
the last inequality is a consequence of (\ref{expineq}). Hence, $b_1(t)=\bigo(|t|)$
as $t\to0$. By induction on the relation (\ref{bineq2}) we get that $b_n(t)=\bigo(|t|)$
as $t\to 0$, for any $n$. Hence, if we set
\[ \delta_N(T)= \sup_{-T\leq t\leq T}\frac{b_N(t)}{|t|},  \]
then this quantity is finite. 

Now, for any $|t|\leq T$, 
\begin{equation}\label{holds}  
b_N(t)\leq |t|\delta_N(T)\quad\mbox{and}\quad b_N(B_Nt)\leq B_N|t|\delta_N(T), 
\end{equation}
where the last inequality follows from the first and the fact that $B_N\in(0,1)$. 

Now, define, for $k\geq N$, \[ \delta_{k+1}(T)=B_k\delta_k(T)+h_k(T)/g_k. \]  Then, 
if we assume that (\ref{holds}) holds for $k$ in place of $n$,
\begin{align*}
 |t|\delta_{k+1}(T)\geq b_k(B_kt)+|t|h_k(T)/g_k\geq b_{k+1}(t),
\end{align*}
where the last inequality is due to relation (\ref{bineq2}). As a consequence, since
$B_{k+1}\in(0,1)$ we also get $b_{k+1}(B_{k+1}t)\leq|t|B_{k+1}\delta_{k+1}(T)$. By induction
$b_k(t)\leq |t|\delta_k(T)$ for all $|t|\leq T$ and $k\geq N$. 

Now, an application of Chung's lemma to $\delta_k(T)$ together with (\ref{sqrte}) reveals that 
$\limsup_k \delta_k(T)\leq \sqrt\epsilon T^2C_V/6\Gamma=\epsilon'$. As this works for every $\epsilon'$
we conclude that $\delta_n(T)$ and thus $b_n(t)$ tends to zero. 

To conclude this section, let us weaken assumption (\ref{kraf}) to (\ref{weakkraf}). Then 
instead of (\ref{sqrte}) we would have 
\begin{equation*} \lim_{n\to\infty}\frac 1n\sum_{k=1}^n h_k(t)=t^2\sqrt\epsilon\bigo(1) \end{equation*}
and we would apply Corollary \ref{Vaclav} instead of Chung's lemma in the preceding paragraph
with the same conclusion.

\rymd
\noindent\textbf{Part 2.}
Let $\hat Y_n$ denote a process that satisfies the assumptions of the theorem,
evolving via $\hat Y_n=(1-\hat\Gamma_n/g_n)\hat Y_{n-1}+V_n/\sqrt{g_n}$ 
with arbitrary $\hat Y_1$ and $\hat\Gamma\to\Gamma$ a.s. By
 Corollary \ref{egocorr}, given any 
$\delta>0$, there is an adapted and uniformly convergent sequence $\Gamma_n\to \Gamma$, 
that equals $\hat \Gamma_n$ on a set $B_\delta$ of probability at least $1-\delta$.
Hence, if we define $Y_1=\hat Y_1$ and $Y_n=(1-\Gamma_n/g_n) Y_{n-1}+V_n/\sqrt{g_n}$,
then $\hat Y_n$ and $Y_n$ also agree on $B_\delta$. 

 Below, we will show that $Y_n$ converges
to  N$(0,\sigma^2/2\Gamma)$, regardless of $\delta$. Hence, Lemma \ref{remma1}
gives us the convergence of $\hat Y_n$ to the aforementioned distribution.

Let $Z_n$ evolve according to 
\begin{equation}\label{theZ} Z_{n+1}=(1-\Gamma/g_n)Z_n+[V_{n+1}-\ex_n V_{n+1}]/\sqrt {g_n}, \end{equation}
with $Z_1=0$. Then, $Z_n$ satisfies the assumptions of Part 1 and hence
$Z_n\limd N(0,\sigma^2/2\Gamma)$. If $\Delta_n=Y_n-Z_n$ converges in
probability to 0 it follows from Cramer's theorem that $Y_n$ also
converges in distribution to N$(0,\sigma^2/2\Gamma)$. We show below 
that $\Delta_n$ converges in $L_1$, which implies convergence in distribution.

Now, $\Delta_n$ can be expressed recursively as
\begin{equation} \label{dee}
\Delta_{n+1}=\left(1-\frac{\Gamma_n}{g_n}\right)\Delta_n + Z_n\frac{\Gamma-\Gamma_n}{g_n} +\frac{\ex_n V_{n+1}}{\sqrt{g_n}}.  
\end{equation}

Fix a positive $\epsilon<\Gamma/2$. We want to show that $\limsup\ex|\Delta_n|$ is smaller than 
some constant times  $\epsilon$. 
We consider $n\geq N$ with $N$ large enough so that $g_n>\Gamma-\epsilon$ and 
$|\Gamma_n-\Gamma|<\epsilon$, the latter can be done since $\Gamma_n$ is uniformly convergent.
Hence, from (\ref{dee}), 
we may express the absolute value of $\Delta_{n+1}$ as
\begin{equation}\label{dee2}
 |\Delta_{n+1}|=\left(1-\frac{\Gamma-\epsilon}{g_n}\right)|\Delta_n| + |Z_n|\frac{\epsilon}{g_n}+\lilo(g_n^{-1}) +D_n,
\end{equation}
where $D_n\leq 0$ and the $\lilo$-term comes from assumption $|\ex_n V_{n+1}|=\lilo(1/\sqrt{g_n})$.
  We want to show that $\limsup_n \ex|\Delta_n|$ can be made arbitrarily small, so to proceed we need
a bound on $\ex|Z_n|$. To that end, we calculate from (\ref{theZ})
\begin{equation}\label{thesquare} 
Z_{n+1}^2=\left(1-\frac{2\Gamma-\Gamma/g_n}{g_n}\right) Z_n^2+\frac{\tilde V_{n+1}^2}{g_n}+
 \frac2{\sqrt{g_n}}\left(1-\frac{\Gamma}{g_n}\right) Z_n\tilde V_{n+1}, 
\end{equation}
  where $\tilde V_n=V_n-\ex_{n-1} V_n$. 

By first taking conditional expectation with respect to $\fa_n$ on (\ref{thesquare}) and then
taking expectation, we get  
\[ \ex Z_{n+1}^2=\left(1-\frac{2\Gamma-\Gamma^2/g_n}{g_n}\right)\ex Z_n^2+\frac{\ex \tilde V_{n+1}^2}{g_n},\]  
so that Chung's lemma yields $\limsup_n \ex Z_n^2\leq C_V/2\Gamma$. From the Cauchy-Schwarz inequality
we conclude that $\limsup_n \ex |Z_n|\leq \sqrt{C_V/2\Gamma}$. 

Now, If we take expectation on (\ref{dee2}) and apply Chung's lemma we get
\[ \limsup_n\ex|\Delta_n|\leq \epsilon \sqrt{C_V}/(\sqrt{2\Gamma}(\Gamma-\epsilon)).\]
 Since $\epsilon$ 
was arbitrary, we conclude $\ex|\Delta_n|\to 0$.
\end{proof}

\subsection{Applications to stochastic approximation algorithms} \label{atsaa}
In this section we discuss how Theorem \ref{st} can be applied to stochastic approximation
algorithms, as in Definition \ref{def}. Recall from section \ref{prel} that if $X_n$ is a  
stochastic approximation algorithm and 
 if $Q_n=X_n-p$, where $p$ is a stable point of $f$,  
then
\begin{equation} \label{theQ}
Q_{n+1} =\left[1-\frac{\hat\gamma_{n+1}}{n+1}\right]Q_n+\frac{\hat U_{n+1}}{n+1},
\end{equation}
where 
\begin{equation}\label{hattar}
 \hat\gamma_{n}=n\gamma_{n}h(X_{n-1})\quad\mbox{and}\quad
\hat U_{n}=n\gamma_{n}U_{n}\end{equation}
and $h(x)=-f(x)/(x-p)$ is  
nonnegative close to $p$.

Now, we may assume that $p$ is such that 
$0<\pr(X_n\to p)=\pr(Q_n\to 0)$, see Theorem 4 of \cite{lic} for necessary assumptions for this
to hold. Conditional on the event $\{Q_n\to 0\}$ we want to know how to normalize $Q_n$ to get 
a nontrivial asymptotic distribution.

To that end, let $x,y\in\real$ and define $w(n)=(n+1)^x[\ln (n+1)]^y, n\geq 1$, then by Taylor expanding
we get, for $n\geq 2$,
\begin{align*}
 \frac{w(n)}{w(n-1)}&=\left(1+\frac1n\right)^x\left(1+\frac{\ln(1+1/n)}{\ln n}\right)^y\\
&=\left(1+\frac xn+\bigo(1/n^2) \right)\left(1+\frac y{n\ln n}+\lilo(1/n^2)\right) \\
&=1+\frac xn+\frac y{n\ln n}+\bigo(1/n^2).
\end{align*}
And thus, with $Z_n=w(n)Q_n$,
\begin{align} 
 Z_{n} &= w(n-1)Q_{n-1}\left(1-\frac{\hat\gamma_{n}}{n}\right)\frac{w(n)}{w(n-1)}+\frac{w(n)}{n}\hat U_{n} \nonumber \\
&=\left(1-\frac{\hat\gamma_n-x}{n}+\frac{y}{n\ln n}+\bigo(1/n^2)\right)Z_{n-1}
+\frac{V_{n}}{n^{1-x}[\ln n]^{-y}}, \label{W_n}
\end{align}
where 
\begin{align}
 V_n &=\delta_{x,y,n} \hat U_n 
,\quad \mbox{and}  \label{veen} \\
\delta_{x,y,n} &=\left(\frac {n+1}{n}\right)^x\left(\frac{\ln (n+1)}{\ln n}\right)^y,\quad n\geq 2. \label{dElTa}
\end{align}

Assume that $\hat\gamma_n$ tends to a nonnegative real number $\hat\gamma$.
In order for (\ref{W_n}) to fit Theorem \ref{st} we need either 
\begin{romenumerate}
 \item $\hat\gamma-x>0$, $y=0$ and $x=1/2$, or 
 \item $\hat\gamma-x=0$, $y=-1/2$ and $x=1/2$. 
\end{romenumerate}

\newpage

Thus, 
\begin{itemize}
\item[(i)] when $\hat\gamma_n\to\hat\gamma>1/2$, we
consider $Z_n=\sqrt nQ_n$ which satisfies 
\begin{align*}
Z_n &=\left(1-\frac{\hat\gamma_n-1/2+\bigo(1/n)}{n}\right)Z_{n-1}
+\frac{\delta_n\hat U_n}{\sqrt n},
\end{align*}
where $\delta_n=\delta_{1/2,0,n}$, and  thus $g_n=n$.
\item[(ii)] When $\hat\gamma_n\to1/2$, we consider $Z_n=\sqrt{\frac n{\ln n}}Q_n$ which satisfies
\begin{align*}
Z_n &=
\left(1-\frac {1/2+[\hat\gamma_n-1/2+\bigo(1/n)]\ln n}{n\ln n}\right)Z_{n-1}
+\frac{\delta_n'\hat U_{n}}{\sqrt{n\ln n}},
\end{align*}
where $\delta_n'=\delta_{0.5,-0.5,n}$ and thus $g_n=n\ln n$. In this case we must verify 
that $(\hat\gamma_n-1/2)\cdot\ln n\to 0$ a.s. 
\end{itemize}
Note that the positive sequence $\delta_{x,y,n}$ satisfies 
$\sqrt{3/2}\geq \delta_{x,y,n}\to 1$, when $n\geq 2$ and $(x,y)\in\{(1/2,0),(1/2,-1/2)\}$.
Hence, from (\ref{veen}), (\ref{hattar}) and Definition \ref{def},
\begin{align}
 |\ex_n V_{n+1}| &= \delta_{x,y,n}n|\ex_n \gamma_{n+1}U_{n+1}|\leq \frac{\sqrt{3/2}K_e}{n},\quad\mbox{and} \\
 |V_{n+1}| &\leq \sqrt{3/2}c_uK_u, 
\end{align}
and thus $V_n=\delta_{x,y,n}\gamma_nU_n$ satisfies the first and third 
condition listed in (\ref{krav}). 

In application to a specific stochastic approximation algorithm we must make sure
that $\ex_n V_{n+1}^2$ tends to some (strictly) positive constant, that 
$\hat\gamma_n\to \hat\gamma\geq 1/2$ and, if $\hat\gamma=1/2$, that $\ln n\cdot(\hat\gamma_n-1/2)\to 0$ a.s.

\subsection{Applications to generalized P\'olya urns} \label{genpol} 
In this section we apply Theorem \ref{st} to the  generalized Pólya urn model described in the introduction
and defined by the replacement matrix (\ref{repmat}).
Asymptotic normality 
(as well as general limit theorems) is well studied for generalized Pólya urn models (see e.g.\ 
\cite{Fre65}, \cite{BP85}, \cite{Gou93},\cite{Smy96}, \cite{Jan04}, \cite{Jan06}) 
so we do not expect these results to be new.

Recall that the fraction of white balls $X_n$ in this model, when $\min\{a+b,c+d\}>0$, is a stochastic approximation
algorithm with drift function 
$f(x)=\alpha x^2+\beta x+c$, where 
$\alpha=c+d-a-b$ and $\beta=a-2c-d$. The error function $\e(Z_n)=\ex_n U_{n+1}^2$ is given by 
$\e(x)=x(1-x)[a-c+\alpha x]^2$. The total number of balls at time $n$ is denoted $T_n$.

Below, we give calculate explicitly the parameters of the limiting normal distribution 
in the case of $\alpha=0$, and give a brief discussion on the case $\alpha\neq 0$.

\subsubsection{The case $\alpha=0$} 
$\alpha$ is zero exactly when $a+b=c+d=:T$, which we assume positive. This has the added benefit that 
\[ \gamma_{n}=\frac1{T_0+nT}\]  are
deterministic with $n\gamma_n\to \gamma= 1/T$ and that $\e(x)=x(1-x)(a-c)^2$, i.e.\ the variance of 
$U_n$ never vanishes, except at the boundary, as long as $a\neq c$ (which would imply also
that $b=d$ which makes the process completely deterministic). Hence, we must demand $a\neq c$. 

Note that we have
\[ -h(x):=f'(x)=\beta=a-2c-d=-b-c\leq 0, \] so as long as $c+b\neq 0$, any zero of $f$ is 
stable. We are looking for a $p\in(0,1)$ such that $f(p)=0$. Since $p=c/(c+b)$ we must demand 
$c>0$ and $b>0$.  
Now, $h(x)=c+b$ so with $\hat\gamma=\gamma h(p)$ we get
\[ \hat\gamma=\frac{b+c}{a+b} \quad\mbox{and}\quad \hat\gamma-\frac12=\frac12\cdot\frac{b+2c-a}{a+b} \]
and thus $\hat\gamma> 1/2$ if $a< b+2c$ and $\hat\gamma=1/2$ if $a=b+2c$. 

The $\sigma^2$ of Theorem \ref{st} corresponds to 
\[ \lim_n \ex_{n-1} [(n\gamma_nU_n)^2]= \frac{\e(p)}{T^2}=\frac{bc(a-c)^2}{(a+b)^2(b+c)^2}. \]

So,
if $c\neq a<b+2c$, and $b,c>0$, then 
\begin{equation} \label{min}
\sqrt n\left(X_n-\frac{c}{c+b}\right) \limd \text N\left(0,\frac{\e(p)/T^2}{2(\hat\gamma-1/2)}\right)\,\mbox{i.e.}\;
\text N\left(0,\frac{bc(a-c)^2}{(a+b)(c+b)^2(b+2c-a)}\right).
\end{equation}

If $a=b+2c$ then $\hat\gamma=1/2$ and we first need to verify that $\hat\gamma_n-1/2$ tends to zero
faster than $\ln n$. That this is true is shown by direct calculation;
\begin{align}  
n\gamma_nh(X_{n-1})-\frac 12 &=n\frac{b+c}{T_0+n2(b+c))}-\frac12 \label{gjort1} \\
&=\frac 1n\cdot\frac{-T_0}{2(b+c)(2+T_0/n(b+c))}=\lilo(1/\ln n) \label{gjort2}.
\end{align}
The  variance in the central limit theorem is
\[ \frac{\e(p)T^{-2}}{2\cdot1/2}=\frac{bc(a-c)^2}{(a+b)^2(c+b)^2}=\frac{bc}{4(b+c)^2}, \]
since $a=b+2c$.
Thus,
\begin{equation} \label{lik}
\sqrt{\frac{n}{\ln n}}\left(X_n-\frac{c}{c+b}\right)\limd \text N\left(0,\frac{bc}{4(b+c)^2}\right),
\end{equation}
when $b,c>0$ and $a=b+2c$.

\begin{exempel}[Friedman's urn] \rm The urn process with replacement matrix
\begin{align*}  
\left( \begin{array}{c c}
a & b \\ b & a 
       \end{array} \right),
 \end{align*}
where $b>0$, is commonly known as Friedman's urn. The fraction of white balls $X_n$ is
a stochastic approximation 
algorithm with drift function $f(x)=-2b(x-1/2)$. It is straightforward to verify from 
(\ref{min}) and (\ref{lik}) that
\begin{itemize}
 \item[(i)] $3b>a$ implies \[ \sqrt n(X_n-1/2)\limd  \text N\left(0,\frac{(a-b)^2}{4(a+b)(3b-a)}\right),\quad\mbox{and} \]

\item[(ii)] $3b=a$ (when $a>0$) implies 
\[ \sqrt{\frac{n}{\ln n}}(X_n-1/2)\limd N(0,1/16), \]
respectively.
\end{itemize}
\end{exempel}

\subsubsection{A remark on the case $\alpha\neq 0$} \label{alfaejnoll}
To write down the general formula is rather cumbersome, so lets look  at just one example
before making a general comment. 
\begin{exempel} [Toy example] \rm The fraction of white balls $X_n$ evolving in accordance with the replacement
matrix
 \begin{align*}  
\left( \begin{array}{c c}
4 & 5 \\ 3 & 2 
       \end{array} \right),
 \end{align*}
has a drift function $f(x)=-4x^2-4x+3=-4(x+3/2)(x-1/2)$, and thus 
the stable zero is $1/2$ and $h(1/2)=8$. Then
$n\gamma_n$ converge to $[(4+5)\frac 12+(3+2)\frac 12]^{-1}=1/7$ and
thus $\hat\gamma=8/7>1/2$. Since, $\e(1/2)=\frac14(1-4\frac 12)^2=1/4$
we know
\[ \sqrt n(X_n-1/2)\limd \mathrm{N} \left(0,\frac{\e(1/2)7^{-2}}{2(8/7-1/2)}\right),\;\mbox{i.e.}\; \mathrm{N}(0,1/252). \]
\end{exempel}
For any given replacement matrix that has $\hat\gamma>1/2$ and non-vanishing error
terms at $p$ (see Remark \ref{singular} for an exception)
it is clear that a central limit theorem applies and
the parameters are not too difficult to calculate. When  $\hat\gamma=1/2$ it is not a priori clear that
$\hat\gamma_n-1/2$ is $\lilo(1/\ln n)$, which must hold for the central limit
theorem to apply. When the step lengths are deterministic (the case $\alpha=0$) this followed from
the calculation (\ref{gjort2}). When they are not, this fact will follow
 from the assertion $\hat\gamma_n-\hat\gamma=\bigo(|X_n-p|+1/n)$, in section 
\ref{limappgenpol}, and Lemma \ref{fasterthanln}, both below, since taken together
these facts imply that $\hat\gamma_n-1/2=\lilo(n^{-\beta})$ for any $\beta<1/2$.

\newpage

\section{A limit theorem} \label{alt}
We present here a limit theorem for when the parameter $\hat\gamma_n$, defined
by (\ref{hattar}) and Definition \ref{def}, tends to 
a limit in $(0,1/2)$.  
We recall that $X_n$ is a stochastic approximation algorithm 
according to Definition \ref{def}, that $p$ is a stable zero of the drift function,
$Q_n=X_n-p$ and that it is convenient to write the recursive evolution of 
$Q_n$ in the form of (\ref{theQ}). 

A corresponding limit theorem for the Robbins-Monro algorithm can be found in \cite{MP73}, 
and we follow their approach.  

\begin{theorem} \label{genlim}
Suppose $X_n$ is a stochastic approximation algorithm, according to Definition \ref{def}, 
with drift function $f$ having a stable point $p$. Assume that 
$\{ X_n\to p\}$ and that for some $\alpha\in(0,1/2)$ we have a.s.\ 
\begin{equation}\label{gammas} \hat\gamma_{n}-\alpha=\bigo(|X_{n}-p|+1/n),\end{equation}
where $\hat\gamma_n=-n\gamma_nf(X_{n-1})/(X_{n-1}-p)$. 

Then $n^\alpha (X_n-p)$ converges a.s.\ to a random variable. 
\end{theorem}

\begin{rem} \label{as_gammas}
Recall that $h(x)=-f(x)/(x-p)$. In applications of Theorem \ref{genlim},
one can try to verify assumption (\ref{gammas}) by verifying $h(X_n)-h(p)$ 
and $n\gamma_n-\gamma$ to be $\bigo(|X_{n}-p|+1/n)$ separately. Notice that
$h(x)-h(p)=\bigo(|x-p|)$ if e.g.\ $f$ is twice differentiable at $p$, which is
the case for the generalized P\'olya urn process described in the introduction. 
That $n\gamma_n-\gamma=\bigo(|X_n-p|+1/n)$ for that particular process 
is shown in section \ref{limappgenpol}. 
\end{rem}

\begin{proof}
By first rewriting  (\ref{theQ}) as
\[ Q_n=\left(1-\frac\alpha n\right)Q_{n-1}+\frac{\hat U_n}{n}+\frac{(\alpha-\hat\gamma_n)Q_{n-1}}{n} \]
and then iterating, we get
\begin{equation} \label{qq}
n^\alpha Q_n =  Q_0n^\alpha P_\alpha(1,n)+G_n+F_n,
\end{equation}
where $P_\alpha(m,n) =\prod_{k=m}^n\left(1-\frac\alpha n\right)$ (which equals 1 if $m>n$),
\begin{align*}
G_n =\sum_{k=1}^n\frac 1k\hat U_k n^\alpha P_\alpha(k+1,n) \quad\mbox{and}\quad
F_n =\sum_{k=1}^n\frac 1k(\alpha-\hat\gamma_k)Q_{k-1}n^\alpha P_\alpha(k+1,n).
\end{align*}
Recall that Lemma \ref{prodlemma} states that $P_\alpha(m,n)$ is $(m/n)^\alpha(1+\bigo(1/m))$ and
that $n^\alpha P_\alpha(m,n)$ is convergent, as $n\to \infty$.
Thus, the first term on the right hand side of (\ref{qq}) is convergent. 
The second term $G_n$, equals, by the definition of
$\hat U_k$ in (\ref{hattar}) and Definition \ref{def},
\[ G_n=\sum_{k=1}^n k^\alpha\gamma_k U_kl_{k,n}, \] where 
\begin{equation} \label{lkn}
l_{k,n}=(n/k)^\alpha P_\alpha(k+1,n). 
\end{equation} 
The limit 
\begin{equation} \label{lk}
l_k=\lim_n l_{k,n} 
\end{equation} 
exists and is uniformly bounded by Lemma \ref{prodlemma}. The quantity
$G_n^*=\sum_1^n k^\alpha\gamma_k U_kl_k$ will be a.s.\ convergent by
Lemma \ref{improve}, since $l_k$ is bounded, $|k^\alpha\gamma_kU_k|^2=\bigo(k^{2-2\alpha})$,
$|\ex_{k-1}[k^\alpha\gamma_kU_k]|=\bigo(k^{\alpha-2})$ and $\alpha<1/2$.

By Lemma \ref{prodlemma} it is easy to see that $l_{k}/l_{k,n}=1+\bigo(n^{-1})$. Since
$l_{k,n}$ also is bounded, it follows that $l_k-l_{k,n}=\bigo(n^{-1})$. Hence, there is
some constant $C$ such that 
\[ |G_n^*-G_n|\leq \sum_{k=1}^n |k^\alpha\gamma_k U_k(l_k-l_{k,n})| \leq C\frac 1n\sum_{k=1}^n k^{\alpha-1}, \]
which tends to 0, as $n\to\infty$, and thus implies the a.s.\ convergence of $G_n$. 

By squaring relation (\ref{theQ}), taking expectations, using the bounds of $\gamma_k$ and
$U_k$, as well as smoothing, we get
\begin{align} &\ex  Q_k^2=
 \ex \left[\left(1-\frac{2\hat\gamma_k}{k}+\frac{\hat\gamma_k^2}{k^2}\right) Q_{k-1}^2 \right]
 +\frac{\ex [\hat U_k^2]}{k^2} +\frac {2\ex \left[(1-\hat\gamma_k/k)Q_{k-1}\hat U_k\right]}k \nonumber \\
&\leq \ex\left[\left(1-\frac{2\hat\gamma_k-\hat\gamma_k^2/k}{k}\right) Q_{k-1}^2\right]
 +\frac{c_u^2K_u^2}{k^2} +\frac 2k\ex\left\{|Q_{k-1}|\left|\ex_{k-1} \left[\hat U_k-
  \hat\gamma_k\hat U_k/k)\right]\right|\right\}. \label{tobc}
\end{align}

Next, make two \emph{extra} assumptions. First that $\hat\gamma_n\to \alpha$ uniformly. Then, given any $\epsilon\in(0,\alpha)$
we can find a $N_\epsilon$ such that $k\geq N_\epsilon$ implies $2\hat\gamma_k-\hat\gamma_k^2/k\geq 2\alpha-\epsilon$
Second,  make  the assumption that $\hat\gamma_k-\alpha=\bigo(|Q_k|+1/k)$ more restrictive by
assuming that  
\begin{equation}\label{Lk} 
 \hat\gamma_k-\alpha=L_k(|Q_k|+1/k)
\end{equation}
 for a \emph{bounded} (stochastic) sequence $L_k$.

So, assuming $k\geq N_\epsilon$ and noting that $Q_{k-1}h(X_{k-1})=-f(X_{k-1})$, we can continue 
\begin{align*}
 (\ref{tobc}) &\leq \left(1-\frac{2\alpha-\epsilon}{k}\right)\ex Q_{k-1}^2+
\frac{2}{k^2}\left(c_u^2K_u^2+K_e+c_u^2K_uK_f \right),
\end{align*}
which, by Lemma \ref{Venter}, implies that $\ex Q_k^2=\bigo(k^{-2\alpha+\epsilon})$. 

Now, we are prepared for the third and last term of the right hand side of (\ref{qq}). 
Below,
we show that $F_n$ converges. Notice, however, that this also works for
$\alpha=1/2$,  a fact we exploit in the proof of Lemma \ref{fasterthanln} below.
\begin{align} 
F_n &=\sum_{k=1}^n\frac{1}{k^{1-\alpha}}Q_{k-1}(\alpha-\hat\gamma_k)(n/k)^\alpha P_\alpha(k+1,n)  \nonumber \\
&=-\sum_{k=1}^n \frac{1}{k^{1-\alpha}}Q_{k-1}L_k(|Q_k|+1/k)l_{k,n} \nonumber \\
&=\sum_{k=1}^n \frac{1}{k^{1-\alpha}}Q_{k-1}Q_kL_k'l_{k,n} - \sum_{k=1}^n \frac{1}{k^{2-\alpha}}Q_{k-1}L_kl_{k,n}, \label{secsum}
\end{align}
where $l_{k,n}$ and $L_{k}$ are  defined in (\ref{lkn}) and (\ref{Lk}), respectively,
and where $L_k'=-\sgn(Q_k)L_k$. Similarly to how we
showed convergence of $G_n$, we compare the  second sum in (\ref{secsum}) with
$\sum_1^nQ_{k-1} L_kl_k/k^{2-\alpha}$, with $l_k$ defined in (\ref{lk}). The infinite sum
is absolutely convergent, since $|L_k|$, $|l_k|$ and $|Q_k|$ are bounded and hence the sequence
of partial sums converges. The absolute difference between this sum and the second sum in (\ref{secsum})
is bounded by some constant times $\frac 1n\sum_1^n 1/k^{2-\alpha}$ which tends to zero. Thus,
the second sum  in (\ref{secsum}) converges.

 The first sum in (\ref{secsum}), is by relation (\ref{theQ}), equal to
\begin{align}
 \sum_{k=1}^n \frac{1}{k^{1-\alpha}}L_k'l_{k,n}\left(1-\frac{\gamma_{k}}{k}\right)Q_{k-1}^2+
\sum_{k=1}^n \frac{1}{k^{2-\alpha}}L_k'l_{k,n}\hat U_kQ_{k-1}, \label{secsum2}
\end{align}
where we once again compare the second sum with that we get when replacing $l_{k,n}$ with $l_k$. This
altered sum will be absolutely convergent, since $|L_{k}'|$, $|Q_k|$ and $|\hat U_k|$ are bounded. 
The absolute difference between the altered sum and the original is some constant times
$\frac1n\sum_1^n 1/k^{2-\alpha}$, which tends to zero.

Next, we compare
$T_n:=\sum_1^n |(1-\hat\gamma_k/k)L_k'l_{k}Q_{k-1}^2/k^{1-\alpha}|$  with the first sum in (\ref{secsum2}). 
Since the summands are positive $T_n$ is 
increasing and thus  $T:=\lim_n T_n$ exists, although it may be $\infty$. By Beppo-Levi's theorem, 
\[ \ex T=\sum_{k=1}^\infty\frac{\ex[|L_k'l_k||1-\hat\gamma_k/k|Q_{k-1}^2]}{k^{1-\alpha}}, \]
which must be finite, since $|L_k'|$ and $l_k$ are bounded, $\ex Q_k^2=\bigo(1/k^{2\alpha-\epsilon})$
and $1-\hat\gamma_k/k\leq 1-(2\alpha-\epsilon)/n<1$ if $k>N_\epsilon$.
Then $T<\infty$ a.s.\ since $\pr(T=\infty)>0$ would imply $\ex T=\infty$. 
Yet again, the absolute difference between the first sum in (\ref{secsum2}) and $T_n$ is
some constant times $\frac 1n\sum_1^n1/k^{1-\alpha}$ which tends to zero.

So, $F_n$ is convergent  under the  extra assumptions that 
$\hat\gamma_n\to\alpha$ uniformly and that (\ref{Lk}) is valid for
a bounded $L_k$, neither which are assumptions of the theorem. However, we
know that, given any $\delta>0$, by Corollary \ref{egocorr} there exists an adapted and uniformly convergent sequence
$\hat\gamma_n^{**}\to \alpha$ such that the sequences $\hat\gamma_n$ and $\hat\gamma_n^{**}$ agree on a set 
of probability at least $1-\delta$. 

Also, we have assumed that $\hat\gamma_n-\alpha=\bigo(|Q_n|+1/n)$ a.s.\ so 
there is a random variable $L$ such that 
\[ L_n=\frac{\hat\gamma_n^*-\alpha}{|Q_n|+1/n}, \]
has the property $|L_n|\leq L$. Since $L<\infty$ a.s.\ there must be
a $C_\delta$ such that $\pr(L>C_\delta)<\delta$. Define the adapted $\hat\gamma_n^{*}$ by
\[ \hat\gamma_n^{*}(\omega)=
 \begin{cases}
  \hat\gamma_n^{**}(\omega), & \mbox{if\;} L_n\leq C_\delta \\
  \alpha, & \mbox{if\;} L_n> C_\delta. 
 \end{cases} \]

Therefore, a process $Q_n^*$ defined by (\ref{theQ}) -- with $Q_n^*$
instead of $Q_n$ and $\hat\gamma_n^*$ instead of $\hat\gamma_n$ -- will satisfy the above argument. But then,
as $Q_n$ agree with $Q_n^*$ on a set of probability at least $1-2\delta$, the probability that $Q_n$ fails to
converge must be less than $2\delta$. But since $\delta$ is arbitrary, this must in fact be zero.
\end{proof}

By making some small adjustment in the proof of Theorem \ref{genlim}, we get the following lemma,
which is only needed in establishing - together with Section \ref{limappgenpol} - that when 
considering the application to the generalized P\'olya urn and $\hat\gamma_n\to 1/2$ 
we have $\hat\gamma_n-1/2=\lilo(1/\ln n)$, as remarked at the end of Section \ref{alfaejnoll}.
\begin{lemma}\label{fasterthanln}
The same assumptions as Theorem \ref{genlim} but $\alpha=1/2$, implies that
$n^\beta (X_n-p)$ converges a.s.\ to 0, for any $\beta<1/2$.
\end{lemma}

\begin{proof}
 As in the proof of Theorem \ref{genlim}, equation.\ (\ref{qq}) with $\alpha=1/2$, we can write 
\begin{equation}\label{qq2} 
n^\beta Q_n=n^{\beta-1/2}Q_0l_{1,n}+ n^{\beta-1/2} G_n+ n^{\beta-1/2} F_n. 
\end{equation}
The first term on the right hand side of (\ref{qq2}) obviously tends to zero a.s., since
$l_{1,n}$ is convergent and $\beta<1/2$. 

Fix an $\epsilon\in (0,1/2-\beta)$. Write the second term of the right hand side of (\ref{qq2}) as
\begin{equation} \label{less}
n^{\beta-1/2}G_n=\frac1{n^{1/2-\beta-\epsilon}}\sum_{k=1}^nk^{1/2-\epsilon}(k/n)^\epsilon l_{k,n}\gamma_k U_k 
\end{equation}
First, compare the sum in (\ref{less}) with the  sum $\sum_1^n k^{1/2-\epsilon} l_k\gamma_k U_k$,
which by Lemma \ref{improve} is convergent. Then $\sum_{1}^n(k/n)^\epsilon k^{1/2-\epsilon} l_{k}\gamma_k U_k$
converges to 0 by Lemma \ref{Kron}. The absolute difference to 
$\sum_1^n (k/n)^\epsilon k^{1/2-\epsilon} l_{k,n}\gamma_k U_k$ tends to zero, so
this latter sum must also tend to zero, as well as the right hand side of (\ref{less}).

Finally, since $F_n$ is convergent (shown in the proof of Theorem \ref{genlim},
a remark made just before  (\ref{secsum})), certainly
$n^{\beta-1/2} F_n$ will tend to zero.
\end{proof}

\subsection{On the application of Theorem \ref{genlim} to generalized Pólya urns} \label{limappgenpol}
In this section we will verify condition (\ref{gammas}) of Theorem \ref{genlim} for the generalized Pólya
urn considered in Section \ref{genpol} for nonsingular replacement matrices (i.e.\
when the matrix (\ref{repmat}) has $ad\neq bc$, see also Remark \ref{singular}). The singular
case actually has $\hat\gamma=1$, so is not applicable to Theorem \ref{genlim}.
 Since the drift $f$ for such a process is
always twice (in fact, infinitely) differentiable, it suffices, by Remark \ref{as_gammas}, to check that 
$n\gamma_n-\gamma=\bigo(|X_n-p|+1/n)$.  Recall that $\gamma_n=1/T_n$, where 
$T_n$ is the total number of balls in the urn at time $n$ and that $p$ denotes a (stable)
zero of $f$, defined in Section \ref{prel} , i.e.\ 
\begin{equation}
\label{p} \alpha p^2+(a-2c-d)p+c=0. 
\end{equation}
Now, if $\alpha=0$, i.e.\ $a+b=c+d$, then it is easy to show that $n/T_n-1/(a+b)=\bigo(1/n)$ and this
is essentially done in (\ref{gjort2}), so assume $\alpha\neq 0$.
We can write (\ref{p})
as $p[\alpha p+a-c]=(c+d)p-c$. If $\alpha p+a-c=0$, then necessarily $(c+d)p-c=0$ and these
two facts would imply that $(c-a)/\alpha=c/(c+d)$ which implies $bc=ad$, i.e.\ a singular matrix,
which is a case we have excluded from  consideration. Hence, $p\alpha+a-c\neq 0$. Analogously,
one can show that $p\alpha-c-d\neq 0$.

Let $W_n^*$ denote the number of times a white
ball has been drawn, so that $W_n$, the number of white balls, and $T_n$ can be described
by
\begin{align*}
 W_n &= W_0+cn+(a-c)W_n^* \quad\mbox{and} \\
T_n &= T_0+(c+d)n-\alpha W_n^*, \;\mbox{respectively, and $\alpha=c+d-a-b$.}
\end{align*}
Note that $W_n^*/n$ will also converge to $p$.  So, if $n$ is large $W_n^*/n$ is close to
$p$. 

It is straightforward to check that, with $T=\lim_n T_n/n=c+d-\alpha p$,
\begin{equation}\label{theT} 
\frac n{T_n}-\frac{1}{T}=\frac{\alpha(W_n^*/n-p)-T_0/n}{(c+d-\alpha p)(T_0/n+c+d-\alpha W_n^*/n)},  
\end{equation}
From the above discussion, we know that we are not dividing by $0$ in the last equation,
at least not when $n$ is large (which is what matters here). 

Next, $X_n=W_n/T_n$, being also a function of $W_n^*$, can be inversed to yield 
\begin{equation*} \frac{W_n^*}{n}=\frac
 { \frac{T_0X_n-W_0}{n} + (c+d)X_n-c  }  { a-c+\alpha X_n  }, \end{equation*}
where again, we are not dividing by zero (if $n$ is large). From the last equation,
it is straightforward to calculate that 
\[ \frac{W_n^*}{n}-p=\frac{W_n^*}{n}-\frac{p(c+d)-c}{a-c+\alpha p} =C_1(n)(X_n-p)+C_2(n)/n, \]
for  (eventually) bounded sequences $C_1(n)$ and $C_2(n)$ whose precise values are unimportant.  
Hence, in the nonsingular case $ad\neq bc$,
\begin{equation}\label{hatten2} W_n^*/n-p=\bigo(|X_n-p|+1/n). \end{equation}
and by (\ref{theT}) we have
$\gamma_n-\gamma=\bigo(|X_n-p|+1/n)$. 

So, together with Remark \ref{as_gammas}, since $h(x)=-f(x)/(x-p)$ is differentiable 
for the generalized P\'olya urn process, we know that such a process fulfills the
assumption (\ref{qq}). Also, when $\alpha=1/2$, we know know, from the above and Lemma
\ref{fasterthanln}, that $\hat\gamma-1/2$ is $\lilo(n^{-\beta})$ for any $\beta<1/2$ and
hence $\lilo(1/\ln n)$. This settles the query at the end of Section \ref{alfaejnoll}.

\subsection{A note on the problem of non-normal limiting distributions}\label{anotponld}
It is tempting to try to find the limiting distribution for stochastic approximation
algorithms in the case when $\hat\gamma\in(0,1/2)$. Any such result must of course be
applicable to any process that fits into the stochastic approximation
scheme, especially generalized Pólya urns. Limit theorems for this urn model are well studied,
see e.g.\ the references made in Section \ref{genpol}, and it is therefore known that
the limiting distributions can be quite cumbersome.

That the situation is more complicated for $\hat\gamma\in(0,1/2)$ 
 -- as opposed to $\hat\gamma\geq 1/2$ -- is already seen from (\ref{qq}). We would assume e.g.\ that
the distribution depends on the initial condition of the urn, which is not the case when 
$\hat\gamma\geq 1/2$ and the central limit theorem applies.

In the following example we exhibit two processes converging to the same point, 
and for which the parameters  $\gamma$, $h(p)$ and $\sigma^2\;(=\lim_n\ex_n U_{n+1}^2)$   are the same, yet the 
limiting distributions are different, even if we start with identical initial conditions. 
\begin{exempel}
Consider two generalized P\'olya urn processes. Let
$X_n$ and $Z_n$ be the proportion of white balls under the replacement matrices
\begin{equation*}
 \left( \begin{array}{c l}
         4 & 1 \\ 1 & 4
       \end{array} \right)\quad\mbox{and}\quad
\left( \begin{array}{c l}
        3 & 0 \\ 2 & 5
       \end{array} \right),\;\mbox{respectively,} 
\end{equation*}
with otherwise identical initial conditions. 

The drift functions are 
\begin{equation*} 
g_X(x)=-2(x-1/2) \mand g_Z(x)=4(x-1)(x-1/2), 
\end{equation*}
respectively, and hence  both processes will converge to $p=1/2$ a.s.\ (by
Theorem 6 of \cite{lic}).
The normalized step length $n\gamma_n$ for the $Z$ and $X$ process will both tend to $1/5$, 
in the former case this is due to $3\frac 12+(2+5)(1-\frac12)=5$.  Since also 
 $h_X(p)=2$ and $h_Z(p)=-4(\frac12-1)=2$, that both process have $\hat\gamma=2/5$.
The error functions are $\e_X(x)=$ $x(1-x)3^2$ and $\e_Z(z)=z(1-z)[1+4z]^2$, respectively, 
with $\e_Z(p)=\e_X(p)=9/4$. 

Theorem 1.3 (iii) of \cite{Jan06}  gives the asymptotic behavior of the number 
of white balls\footnote{Note that 
the roles of black and white is reversed in \cite{Jan06}. } for the $Z$-process. The result,
translated to the proportion of white balls instead of total number thereof,
is that
\[ n^{2/5}(1/2-Z_n) \limd \frac 1{3\cdot 2^{8/5}}W, \]
where $W$ is a distribution given in terms of its characteristic function. 
Now, $W$ is somewhat elusive, but section 8 of \cite{Jan06} has results 
on some of its properties. More specifically, Theorem 8.2 reveals that 
if the urn initially contains balls of both colors, then 
$\ex|W|<\infty$ exactly when $B_0>3$, and, moreover
\begin{equation} \label{komplicerad}
 \ex W= \frac{1}{\Gamma(B_0/5)}\left( W_0\Gamma((B_0-3)/5)-5\Gamma((B_0+2)/5)  \right). 
\end{equation}

The urn model describing $X$ is known as Friedman's urn. In Theorem 3.1 of \cite{Fre65} one can find the
following result (as a special case)
\[ 10n^{2/5}(1/2-X_n) \limd W', \]
where $W'$ is a random variable, not identified. However, if $W_0=B_0>0$ then $W'$ is symmetric 
about 0 (but \emph{not} normal).    
Of course, this symmetry should come as no surprise since, given symmetrical initial conditions, black and white are 
interchangeable due to the symmetry of the replacement matrix.

Then, as (\ref{komplicerad}) typically is not 0 when evaluated at $B_0=W_0>3$, we can conclude
that  $Z_n$ and $X_n$ in general has different limiting distributions. 
\end{exempel}

We end this section with some remarks concerning situations we have not touched upon in this study.

\begin{rem}\label{limnoll}
Another case that may arise is that $h(p)=0$, i.e.\ that the drift function $f$ of a
stochastic approximation algorithm has a double zero at the stable $p$. 
A know application that has this property is the 
generalized Pólya urn with replacement matrix 
\begin{equation*}
 \left( \begin{array}{c l}
         a & 0 \\ b & a
       \end{array} \right)
\end{equation*}
where $a,b>0$. Then the drift function is $f(x)=b(x-1)^2$. Theorem 
1.3 (iv) of \cite{Jan06} has a result for this case. 
\end{rem}

\begin{rem} \label{singular}
The fraction of white balls in the urn model with singular replacement matrix
\begin{equation*}
 \left( \begin{array}{c l}
         a & b \\ \lambda a & \lambda b
       \end{array} \right), \quad \lambda>0,\;a+b>0,
\end{equation*}
will converge to $p=a/(a+b)$. Here we have $\gamma^{-1}=a+b\lambda=h(p)$ so that
$\hat\gamma=1$ which - if it was not for vanishing variance, i.e.\ $\sigma^2=0$ -
would imply a central limit theorem. Note that the convergence is always monotone,
if $X_0<p$ then $X_n<X_{n+1}<p$ and vice versa if $X_0>p$ (if 
$X_0=p$ then $X_n=p$ for all $n$). 	
\end{rem}

\begin{rem}\label{limnej}
Another quite different problem is if the drift function $f$ is \emph{not} differentiable
at the stable point $p$. Then $h(x)=f(x)/(x-p)$ is not continuous and 
$\hat\gamma_n=n\gamma_n h(X_{n-1})$ may not tend to a limit. The papers \cite{KP95} and
\cite{Ker78} deal with this situation. 
\end{rem}

\vspace{0.3cm}
\noindent\textbf{Acknowledgements} \\
I am very grateful to Svante Janson for valuable input and interesting discussions.

\end{document}